\documentclass[12pt, a4paper]{article}
\usepackage{latexsym,amsfonts,amsmath,amssymb,amsthm,url,graphicx,psfrag,epsfig}
\usepackage{amsmath}\multlinegap=0pt
\usepackage[latin1]{inputenc}
\usepackage{amsthm}
\usepackage{amssymb}
\usepackage[francais,english]{babel}
\numberwithin{equation}{section}
\theoremstyle{plain}
\newtheorem{thm}{Theorem}[section]
\newtheorem{coro}[thm]{Corollary}

\newtheorem{defi}[thm]{Definition}
\theoremstyle{definition}

\theoremstyle{remark}

\newcommand{\R}{\mathbb{R}}

\newcommand{\W}{\mathcal{W}}
\newcommand{\J}{\mathcal{J}}

\newcommand\pref[1]{(\ref{#1})}

\let \eps\varepsilon

\newcommand{\C}{\mathcal C}

\newcommand\clos{\overline}

\DeclareMathOperator{\argmin}{argmin}

\def\<#1,#2>{\left<#1,#2\right>}

\DeclareMathOperator{\Lip}{Lip}

\def\D{{\cal D}}
\def\PP{{\cal P}}

\def\P{{\cal P}}
\newcommand{\dd}{\;{\rm d}}

\newcommand{\id}{{\rm id}}
\renewcommand{\L}{{\mathcal L}}
\newcommand{\V}{{\mathcal V}}
\newcommand{\I}{{\mathcal I}}
\newcommand{\E}{{\mathcal E}}
\newcommand{\LL}{{\mathcal L}}
 
\title{Remarks on existence and uniqueness  of Cournot-Nash equilibria in the non-potential case}
\author {A. Blanchet\thanks{\scriptsize TSE (GREMAQ, Universit\'e de Toulouse), Manufacture des Tabacs
21 all\'ee de Brienne, 31000 Toulouse, FRANCE 
\texttt{Adrien.Blanchet@ut-capitole.fr}},  
G. Carlier\thanks{\scriptsize CEREMADE, UMR CNRS 7534, Universit\'e Paris Dauphine, Pl. de Lattre de Tassigny, 75775 Paris Cedex 16, FRANCE
\texttt{carlier@ceremade.dauphine.fr}}}
\begin{document}
\maketitle
\begin{abstract}
This article is devoted to various methods (optimal transport, fixed-point, ordinary differential equations) to obtain existence and/or uniqueness of Cournot-Nash equilibria for games with a continuum of players with both attractive and repulsive effects. We mainly address separable situations but for which the game does not have a potential, contrary to the variational framework of~\cite{abgc}. We also present several numerical simulations which illustrate the applicability of our approach to compute Cournot-Nash equilibria. 
 \end{abstract}
\section{Introduction}
Equilibria in games with a continuum of players have received a lot of attention since the seminal work of Aumann~\cite{Aumann,Aumann2}, followed by Schmeidler~\cite{Schmeidler} and Mas-Colell~\cite{MasColell}. Following the presentation of Mas-Colell~\cite{MasColell}, we consider a type space $X$ endowed with a probability measure $\mu$. Each agent has to choose an action $y$ from some space $Y$, so as to minimise a cost that depends on her type and action but also on the distribution of strategies resulting from the other agents' behaviour. In this general setting, a Cournot-Nash equilibrium can be viewed as a joint probability measure $\gamma$ on the product $X\times Y$ of the type space and the action space, which gives full mass to pairs $(x,y)$ for which $y$ is cost-minimising for type $x$. What makes the problem involved is the dependence of the cost on the action distribution (that is the second marginal of $\gamma$). This explains that one cannot in general do much better than prove an existence result under regularity assumptions on the cost which are not necessarily realistic, as we shall discuss later.  
\smallskip

Another way to understand the difficulty of the problem lies in the nature of externalities. In realistic examples, there are attractive effects that result in agents choosing similar actions but also repulsive effects (congestion) that result in differentiation. Think of a population of young academics having to decide which research field to work in. Choosing a mainstream or fashionable field might be risky in terms of competition but a novel area is risky too. One expects equilibria to balance the attractive and repulsive effects in some sense, but their structure is not easy to guess when these two opposite effects are present. 
\smallskip

In~\cite{abgc}, we relate Cournot-Nash equilibria to optimal transport theory and identify a class of models which have the structure of a potential game. One then may obtain Cournot-Nash equilibria by the minimisation of a functional over the set of probabilities. One limitation of this approach is that it requires a symmetry in the interaction terms. The goal of the present article is to present various techniques to address the non potential case.  We shall indeed prove below, under a separability assumption, several existence, uniqueness, characterisation results and in some cases design simple numerical methods to compute equilibria.
\smallskip

The article is organised as follows. In Section~\ref{regcase}, we recall Mas-Colell's approach to prove existence of Cournot-Nash equilibria under a regularity assumption on the cost that actually rules out the case of congestion. In Section~\ref{sepcase}, we restrict ourselves to the separable case and recall the link with optimal transport, which we use in~\cite{abgc}. In Section~\ref{uniqmonot}, we prove a uniqueness result under a monotonicity assumption whose importance appeared in the Mean-Field-Games theory of Lasry and Lions \cite{mfg1, mfg2, mfg3}. In Section~\ref{bri}, we adopt a different, more direct, approach by best-reply iteration and identify conditions under which the corresponding operator is a contraction of the space of probability measures endowed with the Wasserstein metric. In Section~\ref{combivar}, after recalling the variational approach of~\cite{abgc}, we combine it with a fixed-point argument in $\LL^p$ to prove a rather general existence result under congestion effect and non-symmetric interactions. We end Section~\ref{combivar} by one-dimensional models for which we characterise equilibria by some ordinary differential equations and give some numerical simulations.\smallskip

{\bf{Notations}:} Throughout the article, the type space $X$ and the action space $Y$ will be assumed to be compact metric spaces. Given a Borel probability measures $m$ on $X$, which we shall simply denote $m\in \PP(X)$, and $T$ a Borel map: $X \to Y$,  the push-forward (or image measure) of $m$ through $T$, is the probability measure $T_\# m$ on $Y$ defined by $T_\# m(B)=m(T^{-1}(B))$ for every Borel subset $B$ of $Y$. The canonical projections on $X\times Y$ will be denoted $\pi_X$ and $\pi_Y$ respectively. For $m_1\in \PP(X)$ and $m_2\in \PP(Y)$, we shall denote by $\Pi(m_1, m_2)$ the set of measures $\gamma\in \PP(X\times Y)$ having $m_1$ and $m_2$ as marginals {\it i.e.} such that ${\pi_X}_\#\gamma=m_1$ and ${\pi_Y}_\#\gamma=m_2$. 
%%%%%%%%%%%%%%%%%%%%%%%
%%%%%%%%%%%%%%%%%%%%%%%
\section{The regular case: existence by fixed-point}\label{regcase}
In this section, we recall the existence of Cournot-Nash equilibria in a regular setting in which one can easily apply a fixed-point argument. What follows is essentially due to Mas-Colell~\cite{MasColell}. We give a short proof for the sake of completeness. We consider that the cost for an agent of type $x$ to choose action $y$ when the distribution of the agents' action is $\nu$ is denoted $C(x,y, [\nu])$. Throughout this section, we suppose that, for every $\nu\in \P(Y)$, $C(.,., [\nu])$ is continuous on $X\times Y$ and that 
\begin{equation}\label{REG}
\nu \mapsto C(.,., [\nu]) \mbox{ is a continuous map from $(\P(Y),\textrm{w}\!-\!*)$ to $(\C(X\times Y),\|\cdot\|_\infty)$}
\end{equation}
where $\textrm{w}\!-\!*$ stands for the weak-* topology on $\P(Y)$. In this setting, Cournot-Nash equilibria are naturally defined as:
%-------------------------------------
\begin{defi}[Cournot-Nash equilibrium]\label{defieqregular}
A \emph{Cournot-Nash equilibrium} consists in a joint probability measure $\gamma\in \P(X\times Y)$ whose first marginal is the fixed measure $\mu\in \P(X)$ and such that, denoting by $\nu$ its second marginal, we have
\begin{equation}\label{defcne}
\gamma \left( \left\{(x,y)\in X\times Y \; : \; C(x,y, [\nu])=\min_{z\in Y} C(x,z,[\nu])\right\} \right)=1.
\end{equation}
\end{defi}
%-------------------------------------
%-------------------------------------
\begin{thm}[Existence of Cournot-Nash equilibrium: the regular case]\label{existencereg}
If \pref{REG} holds then there exists at least one Cournot-Nash equilibrium. 
\end{thm}
%-----------------------------------------------------
%-----------------------------------------------------
\begin{proof}
Let 
\[
K:=\{\gamma\in \P(X\times Y) \; : \; {\pi_X}_\# \gamma=\mu\}\;.
\]
Obviously, $K$ is a convex and weakly-$*$ compact subset of $\P(X\times Y)$. Define for every $\gamma=\mu\otimes \gamma^x\in K$,
\[
F(\gamma):=\{\mu \otimes \eta^x\,:\,   \eta^x \in \P(  {\mathcal {Y}}_\gamma(x) )   \}
\]
where ${\mathcal Y}_\gamma(x)$ denotes the closed set
\begin{equation*}
  {\mathcal Y}_\gamma(x):=\argmin_{y\in Y} C(x,y,[\nu]) \quad \mbox{with} \quad  \nu:= {\pi_Y}_\# \gamma.
\end{equation*}
Note that, for $\gamma$ in $K$, setting  $\nu:= {\pi_Y}_\# \gamma$, and the continuous function 
\[
\varphi_\nu(x):=\min_{z\in Y} C(x,z,[\nu])
\]
then $F(\gamma)$ can also be expressed as 
\[
F(\gamma)=\left\{ \theta \in K  \, : \, \int_{X\times Y} \left\{C(x,y,[\nu])-\varphi_\nu(x)\right\} \dd\theta(x,y)=0\right\}\;.
\]
Hence,  $F$ is clearly a weak-$*$ closed and convex valued set-valued map $K\rightrightarrows K$. 

Let us now prove that $F$ has a weak-$*$ closed graph. Since the weak-$*$ topology is metrisable on $\P(X\times Y)$, it is enough to deal with a sequence $(\gamma_n, \theta_n)_n$ such that $\gamma_n \in K$, $\theta_n \in F(\gamma_n)$, $(\gamma_n)_n$ weakly-$*$ converges to some $\gamma$ and $(\theta_n)_n$ weakly-$*$ converges to some $\theta$ in $K$. Setting  $\nu:= {\pi_Y}_\# \gamma$ and  $\nu_n:= {\pi_Y}_\# \gamma_n$, $(\nu_n)_n$ weakly-$*$ converges to  $\nu$. By~\pref{REG}, $(C(.,., [\nu_n]))_n$ uniformly converges to $C(.,.,. [\nu])$ and $(\varphi_{\nu_n})_n$ uniformly converges to $\varphi_\nu$. We can therefore pass to the limit in 
\[
\int_{X\times Y} \left\{C(x,y,[\nu_n])-\varphi_{\nu_n}(x)\right\} \dd\theta_n(x,y)=0
\]
to deduce that $\theta\in F(\gamma)$. It thus follows from Ky Fan's theorem that $F$ admits a fixed-point~$\gamma$. It is then easy to see that $\gamma$ is an equilibrium with $\nu:= {\pi_Y}_\# \gamma$. 
\end{proof}

The previous result is not fully satisfactory. First, the regularity assumption~\pref{REG} is very demanding since it rules out purely local effects, congestion for instance. There are some extensions to a less regular setting, see e.g.~\cite{Kahn}, but to the best of our knowledge all these extensions require some form of lower semi-continuity so that none of them enables one to cope with a local dependence in the cost. Another drawback of an abstract proof relying on a fixed-point theorem is that it is non-constructive and does not provide a characterisation of the equilibria. 
%%%%%%%%%%%%%%%%%%
%%%%%%%%%%%%%%%%%%
\section{The separable case: connexion with optimal transport}\label{sepcase}
We want to consider costs with a possible local dependence, that is a dependence in $\nu(y)$. In such a case $\nu$ has to be absolutely continuous with respect to some fixed reference measure $m_0$ on the action space $Y$ and $\nu(y)$ has to be understood as the value of the Radon-Nikodym derivative of $\nu$ with respect to $m_0$ at $y$. This is motivated by congestion {\it i.e.} the fact that  more frequently played strategies may be more costly. A natural way to take the congestion effect into account is to consider a term of the form $f(y, \nu(y))$ where $f(y,.)$ is increasing. As soon as one incorporates local congestion effects, Assumption~\pref{REG} is violated and to keep the problem still reasonably tractable, we shall now restrict ourselves to the separable case:
\begin{equation}\label{sep}
C(x,y, [\nu])=c(x,y)+\V[\nu](y)
\end{equation}
where $c\in \C(X\times Y)$ is a transport cost depending only on the agent's type and her strategy, whereas the function $\V[\nu]$ %, which may be defined only $m_0$-a.e. and only for $\nu$'s that are absolutely continuous with respect to $m_0$ in the case of a congestion effect, 
captures an additional  cost created by the whole population of players. The typical case we have in mind is 
\begin{equation}\label{typical}
\V[\nu](y):=f(y, \nu(y))+ \I[\nu](y)
\end{equation}
where $f$ is non-decreasing in its second argument and $\I[\nu]$ is regular in the sense that $\I[\nu] \in \C(Y)$ for every $\nu\in \P(Y)$ with
\begin{equation}\label{REG20}
\nu \mapsto \I[\nu] \mbox{ is a continuous map from $(\P(Y),\textrm{w}\!-\!*)$ to $(\C(Y),\|\cdot\|_\infty)$}.
\end{equation}
Typical regular costs are those given by averages {\it i.e.}  $\I[\nu] (y)=\int_Y \phi(y,z) \dd\nu(z)$ where $\phi$ is continuous. Of course, if the congestion cost $f$ is zero, we are in the regular case in the sense of~\pref{REG20}. Taking the strategy distribution $\nu$ as given, an agent of type $x$ therefore aims to minimise in $y$ the cost $y \mapsto c(x,y)+\V[\nu](y)$. Since the latter need not be a continuous or even lower semi-continuous, the definition of an equilibrium has to be modified as follows:
\begin{itemize}
\item when $\V[\nu]$ is regular let us set $\D:=\P(Y)$,
\item %in the congested case where some reference measure $m_0\in \P(Y)$ is fixed according to which the congestion is measured and 
when $\V[\nu]$ is of the form~\pref{typical}, we define the domain:  
\begin{equation}\label{domainD}
\D:=\left\{\nu \in  \P(Y)\cap \L^1(m_0)\; : \; \int_Y   f(y, \nu(y)) \dd m_0(y) <+\infty\right\}.
\end{equation}
\end{itemize}

Note that when $f$ satisfies the power growth condition:
\begin{equation}\label{growth} 
\frac{1}{C}  (t^{\alpha}-1)   \le  f(y, t) \le C (t^{\alpha}+1)
\end{equation}
for some $\alpha >0$ and $C>0$ and every $(y,t)$ then $\D=\P(Y)\cap \L^p(m_0)$ for $p=1+\alpha$.

As before a Cournot-Nash is a joint type-strategy probability measure $\gamma$ that is consistent with the cost minimising behaviour of agents, in the setting of this section, this leads to the definition:
%----------------
\begin{defi}[Cournot-Nash equilibrium: non-regular case]\label{defieqlocal}
A probability $\gamma\in \P(X\times Y)$ is a \emph{Cournot-Nash equilibrium} if its first marginal is $\mu$, its second marginal $\nu$ belongs to $\D$ and there exists $\varphi\in \C(X)$ such that 
\begin{equation}\label{equipp}
\left\{
  \begin{array}{ll}
    c(x,y)+\V[\nu](y)\geq \varphi(x) \quad\mbox{$\forall x\in X$ and $m_0$-a.e. $y \in Y$}\vspace{.3cm}\\
    c(x,y)+\V[\nu](y)= \varphi(x) \quad\mbox{ for $\gamma$-a.e. $(x,y) \in X \times Y$}
\;.
  \end{array}
\right.
\end{equation}
A Cournot-Nash equilibrium $\gamma$ is called \emph{pure} if it  is  of the form $\gamma=({\rm{id}} ,T)_{\#}\mu$ for some Borel map $T$ : $X\to Y$, that is agents with the same type use the same strategy. 
\end{defi}
%----------------
In the separable case, as noted in \cite[Lemma~2.2]{abgc}, Cournot-Nash equilibria are very much related to optimal transport. More precisely, for $\nu\in \PP(Y)$, let $\W_c(\mu, \nu)$ be the least cost of transporting $\mu$ to $\nu$ for the cost $c$ {\it i.e.} the value of the Monge-Kantorovich optimal transport problem:
\[
\W_c(\mu, \nu):=\inf_{\gamma\in \Pi(\mu, \nu)}  \int_{X\times Y} c(x,y) \dd\gamma(x,y)\;.
\]
Let us denote by $\Pi_o(\mu,\nu)$ the set of optimal transport plans\footnote{Since the admissible set is convex and weakly-$*$ compact, it is obvious that the Monge-Kantorovich optimal transport problem admits solutions. For a detailed account of optimal transport theory, we refer to Villani's textbooks~\cite{villani, villani2}}  {\it i.e.}
\[
\Pi_o(\mu,\nu):=\left\{\gamma\in \Pi(\mu, \nu) \; : \; \int_{X\times Y} c(x,y) \dd\gamma(x,y)=\W_c(\mu, \nu)\right\}\;.
\]
The link between Cournot-Nash equilibria and optimal transport is based on the following straightforward observation: if $\gamma$ is a Cournot-Nash equilibrium and $\nu$ denotes its second marginal then $\gamma\in \Pi_o(\mu, \nu)$. Indeed, if $\varphi\in \C(X)$ is such that~\pref{equipp} holds and if $\eta\in \Pi(\mu, \nu)$ then we have 
\begin{align*}
 \int_{X\times Y} c(x,y) \dd\eta(x,y)&\geq  \int_{X\times Y} \{\varphi(x)-\V[\nu](y)\} \dd\eta(x,y)\\
 &=\int_X \varphi(x) \dd\mu(x)-\int_Y \V[\nu](y) \dd\nu(y)
 = \int_{X\times Y} c(x,y) \dd\gamma(x,y)
 \end{align*}
so that $\gamma\in \Pi_o(\mu, \nu)$.

The above argument also proves that $\varphi$ solves the dual of the Monge-Kantorovich optimal transport problem {\it i.e.} maximises the functional
\[
\int_X \varphi(x) \dd\mu(x) +\int_Y \varphi^c(y) \dd\nu(y)
\]
where $\varphi^c$ denotes the $c$-transform of $\varphi$:
\begin{equation}\label{ctransffi}
\varphi^c(y):=\min_{x\in X} \{c(x,y)-\varphi(x)\}.
\end{equation}
In an euclidean setting, there are well-known conditions on $c$, the so-called \emph{twist} or \emph{generalised Spence-Mirrlees condition}, see~\cite{gcmonge}, and $\mu$ which guarantee that an optimal $\gamma$ necessarily is pure whatever $\nu$ is:
%------------------
\begin{coro}[Purity of the equilibrium]\label{purityofCNE}
Assume that $X=\clos{\Omega}$ where $\Omega$ is some open connected bounded subset of $\R^d$ with negligible boundary, that $\mu$ is absolutely continuous with respect to the Lebesgue measure, that $c$ is differentiable with respect to its first argument, that $\nabla_x c$ is continuous on $\R^d\times Y$ and that it satisfies the twist condition: 
\[\mbox{for every $x\in X$, the map $y\in Y\mapsto \nabla_x c(x,y)$ is injective,}\]
then for every $\nu\in \PP(Y)$, $\Pi_0(\mu, \nu)$ consists of a single element and the latter is of the form $\gamma=({\rm{id}}, T)_\#\mu$. Hence every Cournot-Nash equilibrium is pure and actually fully determined by its second marginal.
\end{coro}
%------------------
Note that, in dimension $1$, the assumptions of Corollary~\ref{purityofCNE} on $c$ roughly amounts to the usual Spence-Mirrlees singe-crossing condition {\it i.e.} the strict monotonicity in $y$ of $\partial_x c$ or the fact that the mixed partial derivative $\partial^2_{xy} c$ has a constant sign.
%%%%%%%%%%%%%%%%%%%%%%%%
%%%%%%%%%%%%%%%%%%%%%%%%
\section{Uniqueness under monotonicity of $\V$}\label{uniqmonot}
In the framework of Mean-Field Games, Lions and Lasry~\cite{mfg2}, established that the monotonicity property of $\nu \mapsto \V[\nu]$ is enough to guarantee uniqueness of the equilibrium. A simple adaptation of their argument gives the following uniqueness result: 
%------------------
\begin{thm}[Uniqueness of the equilibrium under monotonicity]\label{thmuniq1}
If $\nu \mapsto V[\nu]$ is strictly monotone in the sense that for every $\nu_1$ and $\nu_2$ in $\D$, one has 
\[\int_Y (\V[\nu_1]-\V[\nu_2]) \dd(\nu_1-\nu_2)\geq 0\]
and the inequality is strict whenever  $\nu_1\neq \nu_2$,  then, all the equilibria have the same second marginal. 
\end{thm}
%------------------
\begin{proof}
Assume that  $(\nu_1, \gamma_1)$ and $(\nu_2, \gamma_2)$ are two equilibria. Let $\varphi_1$, $\varphi_2$ in $\C(X)$ be such that for $i \in \{1,2\}$ 
\[ 
\V[\nu_i](y)\geq \varphi_i(x)-c(x,y)\,, 
\]
for every $x$ and $m_0$-a.e. $y$ with an equality $\gamma_i$-a.e.. Integrating with respect to $\gamma_i$ and using the fact that $\gamma_i \in\Pi(\mu, \nu_i)$, we obtain for $i \in \{1,2\}$ 
\[
\int_Y \V[\nu_i] \dd \nu_i =\int_X \varphi_i \dd \mu-\int_{X\times Y} c \dd\gamma_i\;,
\]
whereas for $i\neq j$ 
\[
\int_{Y} \V[\nu_i] \dd \nu_j \geq \int_X \varphi_i \dd\mu-\int_{X\times Y} c \dd\gamma_j\;.
\]
Substracting, we obtain 
$$
\int_{Y} \V[\nu_1] \dd (\nu_1-\nu_2)\leq \int_{X\times Y} c \dd(\gamma_2-\gamma_1) \quad\mbox{and}\quad \int_{Y} \V[\nu_2] \dd (\nu_2-\nu_1)\leq  \int_{X\times Y} c \dd(\gamma_1-\gamma_2) \;.
$$
So that
$$
\int_{Y} (\V[\nu_1] -\V[\nu_2])\dd (\nu_1-\nu_2)\le 0\;. 
$$
The monotonicity assumption then ensures that $\nu_1=\nu_2$. 
\end{proof}
Typical examples of strictly monotone maps are given by purely local congestion terms $\V[\nu](y)=f(y, \nu(y))$ with $f$ increasing in its second argument. On the contrary, typical regular non-local terms are not monotone. 

Let us however give an example where the congestion effect dominates the canonical interaction term: consider
\[
\V[\nu](y):=\nu(y)+\int_Y \phi(y, z) \,\nu(z)\dd z\;,
\]
with $\D=\L^2(m_0)$. As a simple application of Cauchy-Schwarz inequality, if 
\[
\int_{Y\times Y} \phi^2(y, z) \dd m_0 \otimes\! \dd m_0 <1
\]
then we have
\[
\int_{Y} (\V[\nu_1]-\V[\nu_2]) \dd(\nu_1-\nu_2)\ge \Vert \nu_1-\nu_2\Vert^2_{\L^2(m_0)}(1-\Vert \phi \Vert^2_{\L^2(m_0\otimes m_0)})\;.
\] 
So that the uniqueness result of Theorem~\ref{thmuniq1} applies in this case. 
%%%%%%%%%%%%%%%%%%%%%%%%%%
%%%%%%%%%%%%%%%%%%%%%%%%%%
\section{Quadratic cost: equilibria by best-reply iteration}\label{bri}
In this section, we adopt a direct approach when $c$ is quadratic and $\V[\nu]$ satisfies some suitable convexity condition.
% so that, given $\nu$, solving $x$-type agents' program is an explicit first-order condition. 
Throughout this section, we assume
\begin{itemize}
\item $X=\clos{\Omega}$, $Y=\clos{U}$, where $\Omega$ and $U$ are some open bounded convex subsets of $\R^d$,
\item the cost is quadratic:
\[
c(x,y):=\frac{1}{2} \vert x-y \vert^2, \; \forall (x,y)\in \R^d\times \R^d
\]
\item $\mu$ is absolutely continuous with respect to the Lebesgue measure on $X$ and has a bounded density,
\item $\V[\nu]$ is a smooth and convex function for every $\nu\in \P(Y)$\footnote{This is the case, for instance, if $\V[\nu]$ has the form 
\[
\V[\nu](y):=\int_Y \phi(y,z) \dd\nu(z)
\]
with $\phi$ smooth and convex  with respect to its first argument.},
\item for every $\nu\in \P(Y)$ and every $x\in X$, the solution of 
\begin{equation}\label{br0}
\inf_{y\in Y} \left\{\frac{1}{2} \vert x-y \vert^2 +\V[\nu](y)\right\}
\end{equation}
belongs to $U$\footnote{This is the case as soon as $\V[\nu]$ fulfils some coercivity assumption and $U$ is chosen large enough.}.
\end{itemize}
In this case the solution of~\pref{br0} satisfies the following first-order condition: 
\[
y=(\id+\nabla \V[\nu])^{-1}(x)\;.
\]
%The resolvent operator $(\id+\nabla \V[\nu))^{-1}$ is a very natural operator in convex analysis where it is known as the proximal operator of $\V[\nu]$. 
If agents have a {\it prior} $\nu$ on the other agents' actions, their cost-minimising behaviour leads to another {\it a posteriori} measure on the action space $Y$, namely
\begin{equation}\label{defdeT}
T\nu:=(\id+\nabla \V[\nu])^{-1}_\# \mu\;.
\end{equation}
Clearly,  $(\gamma, \nu)$ is an equilibrium if and only if $\nu=T\nu$ and $\gamma=(\id, (\id+\nabla \V[\nu])^{-1})_\# \mu$. %(which is also the optimal transport plan between $\mu$ and $\nu$  for the quadratic cost).  
Looking for an equilibrium thus amounts to find a fixed point of $T$.  We shall see some additional conditions that ensure that $T$ is a contraction of $\P(Y)$ endowed with the $1$-Wasserstein distance $\W_1$\footnote{By definition the $1$-Wasserstein distance $\W_1$ between probability measures $\nu_1$ and $\nu_2$ is the least average distance for transporting $\nu_1$ into $\nu_2$:
\[
\W_1(\nu_1, \nu_2):=\inf_{\eta\in \Pi(\nu_1, \nu_2)} \int_{Y\times Y}  \vert y_1-y_2\vert \dd\eta (y_1, y_2)\;.
\]}. Since $(\P(Y), \W_1)$ is a complete metric space, these conditions will therefore imply the existence and the uniqueness of an equilibrium. More importantly, from a numerical point,  this equilibrium can be approximated by the iterates of $T$ applied to any $\nu_0\in \P(Y)$). Our additional assumptions read as : there exists $\lambda >0$, $C\geq 0$ and $M>0$ such that for every $(\nu_1, \nu_2)\in \P(Y)\times \P(Y)$ the following inequalities hold
\begin{equation}\label{condfp1}
D^2 \V[\nu_1] \geq \lambda \;\id \quad \mbox{in $X$} 
\end{equation}
\begin{equation}\label{condfp2}
\det(\id +D^2 \V[\nu_1]) \leq M  \quad\mbox{in $X$} 
\end{equation}
\begin{equation}\label{condfp3}
\int_{Y} \vert \nabla \V[\nu_1](y)-\nabla \V[\nu_2](y)\vert \dd y \leq C \W_1(\nu_1, \nu_2) \;.
\end{equation}
%-------------------------------
\begin{thm}[Convergence of the best-reply iteration scheme]\label{thmcontract}
Under the assumptions of the beginning of the section, if~\pref{condfp1},~\pref{condfp2} and~\pref{condfp3} hold and if 
\begin{equation}\label{condfp4}
M\,C \,\Vert \mu\Vert_{\L^\infty} <1+\lambda 
\end{equation}
then the map $T$ defined by~\pref{defdeT} is a contraction of $(\P(Y), \W_1)$. Therefore for every $\nu_0\in \P(Y)$, the sequence $(T^n \nu_0)_n$ converges to $\nu$ in the distance $\W_1$, that is for the weak-$*$ topology. As a consequence there exists a Cournot-Nash unique equilibrium.% $(\gamma,\nu)=((\id, (\id+\nabla \V[\nu])^{-1})_\# \mu, (\id+\nabla \V[\nu])^{-1})_\# \mu)$.  
\end{thm}
%-------------------------------

\begin{proof}
Let $(\nu_1, \nu_2)$ be in $\P(Y)$. Since  $((\id+\nabla \V[\nu_1])^{-1},(\id+\nabla \V[\nu_2])^{-1})_\# \mu$ belongs to $\Pi(T\nu_1, T\nu_2)$, we first have
\begin{equation}\label{inegw1}
\W_1(T\nu_1, T\nu_2)\leq \int_{X} \left\vert (\id+\nabla \V[\nu_1])^{-1}-(\id+\nabla \V[\nu_2])^{-1}\right\vert \dd\mu\;.
\end{equation}
Now let $x\in X$ and $y_i:=(\id+\nabla \V[\nu_i])^{-1}(x)$, we then write
\[
y_1-y_2=\nabla \V[\nu_2] (y_2)-\nabla \V[\nu_1](y_1)
= \nabla \V[\nu_1] (y_2)-\nabla \V[\nu_1](y_1)+(\nabla \V[\nu_2]-\nabla \V[\nu_1])(y_2)\;.
\]
Taking the inner product with $y_1-y_2$, using~\pref{condfp1} and recalling that $D^2 f \geq \lambda \; \id$ implies that $(\nabla f(y_1)-\nabla f(y_2)) \cdot (y_1-y_2) \geq \lambda \vert y_1-y_2\vert^2$, we obtain
\begin{multline*}
\vert y_1-y_2\vert^2 = (y_1-y_2)\cdot \left( \V[\nu_1] (y_2)-\nabla \V[\nu_1](y_1) + (\nabla \V[\nu_2]-\nabla \V[\nu_1])(y_2)\right)  \\ 
\leq -\lambda \vert y_1-y_2\vert^2 +\vert y_1-y_2\vert \cdot \vert (\nabla \V[\nu_2]-\nabla \V[\nu_1])(y_2) \vert\;.
\end{multline*}
So that
\begin{align*}
\vert ((\id+\nabla \V[\nu_1])^{-1}\!-\!(\id+\nabla \V[\nu_2])^{-1})(x) \vert &= \vert y_1-y_2\vert\\
 &\leq \frac{1}{1+\lambda} \left\vert (\nabla \V[\nu_2]-\nabla \V[\nu_1])(y_2) \right\vert\\
 &= \frac{1}{1+\lambda}  \left\vert (\nabla \V[\nu_2]-\nabla \V[\nu_1])\! \circ \! (\id+\nabla \V[\nu_2])^{-1}(x) \right\vert
\end{align*}
Recalling~\pref{inegw1} and using the fact that $(\id+\nabla \V[\nu_2])^{-1}_\# \mu=T\nu_2$, we then obtain
\begin{equation}\label{inegw12}
\W_1(T\nu_1, T\nu_2)\leq \frac{1}{1+\lambda} \int_{Y} \vert \nabla \V[\nu_2 ]-\nabla \V[\nu_1] \vert \dd T\nu_2\;.
\end{equation}
Now it follows from the fact that $(\id+\nabla \V[\nu_2])_\# T\nu_2=\mu$, the injectivity of $\id+\nabla \V[\nu_2]$ and  the change of variables formula that $T\nu_2$ has a density with respect to the Lebesgue measure, again denoted $T\nu_2$, for $y\in (\id+\nabla \V[\nu_2])^{-1}(X)$ given by 
\[ T\nu_2 (y)=\mu(y+\nabla \V[\nu_2](y))\det (\id +D^2 \V[\nu_2](y)).\]
Finally, using \pref{inegw12}-\pref{condfp2} and \pref{condfp3}, we obtain
\begin{equation*}
 \W_1(T\nu_1, T\nu_2)\leq \frac{\Vert \mu \Vert_{L^\infty} M}{1+\lambda} \int_{Y} \vert \nabla \V[\nu_2 ](y)-\nabla \V[\nu_1](y) \vert \dd y 
 \leq \frac{\Vert \mu \Vert_{L^\infty} MC}{1+\lambda} \W_1(\nu_1, \nu_2)\;.
\end{equation*}
The conclusion thus follows from Assumption~\pref{condfp4} and Banach's fixed point theorem. 
\end{proof}

It may seem difficult at first glance to check the assumptions of Theorem~\ref{thmcontract}. The following result gives a class of examples: consider the case where
\begin{equation}\label{exemplebri}
\V[\nu](y)=V_0(y)+\eps \int_Y \phi(y,z) \dd\nu(z)
\end{equation}
where $\eps>0$ is a scalar parameter, capturing the size of interaction for instance.
%-----------------------------
\begin{coro}[A class of example for Theorem~\ref{thmcontract}]
Assume that $\nu\mapsto V[\nu]$ has the form~\pref{exemplebri} where $V_0$ is a smooth and convex function such that $D^2 V_0 \geq \lambda_0 \,\id$ on $Y$ with $\lambda_0>0$ and $\phi$ is a $\C^2(\R^d\times \R^d)$ function.  If $\eps$ is small enough, then there is a unique Cournot-Nash equilibrium.
\end{coro}
%-----------------------------
\begin{proof} It is enough to check that the map $T$ defined by~\pref{defdeT} satisfies~\pref{condfp1}-\pref{condfp2}-\pref{condfp3}-\pref{condfp4} and apply Theorem~\ref{thmcontract}.

Let $\Lambda_0\geq \lambda_0$ be such that $D^2 V_0 \leq \Lambda_0 \,\id$ on $Y$. It is clear that~\pref{condfp1} and~\pref{condfp2} hold respectively with $\lambda=\lambda_0+O(\eps)$  and $M=(1+\Lambda_0 +O(\eps))^d$. As far as~\pref{condfp3} is concerned, we recall the Kantorovich duality formula for $\W_1$, see~\cite{villani, villani2} for details:  
\[
\W_1(\nu_1, \nu_2):=\sup \left\{\int_{Y}  u \dd(\nu_1-\nu_2)\; : \; u \mbox{ $1$-Lipschitz}\right\}.
\]
Hence, for any Lipschitz continuous function $u$ on $Y$ and any pair of probability measures $m_1$, $m_2$ on $Y$ we have 
$$
\left\vert \int_Y u \dd(m_1-m_2) \right\vert \leq \Lip(u,Y)\; \W_1(m_1, m_2)
$$
where $\Lip(u,Y)$ denotes the Lipschitz constant of $u$ on $Y$. Since for $(\nu_1, \nu_2)\in \P(Y)\times \P(Y)$ and $y\in Y$ we have
\[
\nabla \V[\nu_1](y)-\nabla \V[\nu_2](y) =\eps \int_Y \nabla_y \phi(y,z) \dd(\nu_1-\nu_2)(z)
\]
and $\phi$ is in $\C^2$ with  $\nabla_y \phi$ locally Lipschitz, we obtain
\[ 
\int_{Y} \left\vert \nabla \V[\nu_2 ](y)-\nabla \V[\nu_1](y) \right\vert \dd y \leq\eps \left(\int_Y \Lip(\nabla_y \phi(y,.) \dd y\right) \W_1(\nu_1, \nu_2) \;.
\]
So that~\pref{condfp3} holds with $C=O(\eps)$. Thus~\pref{condfp4} is satisfied for small enough $\eps$.  
\end{proof}

Computing the iterates of the map $T$ is easy so that one can find numerically the equilibrium, as illustrated in Figure~\ref{fig:2d} in dimension $2$. 
%-----------------
\begin{figure}[ht!]
 \begin{minipage}[t]{.5\linewidth}
\centering\epsfig{figure=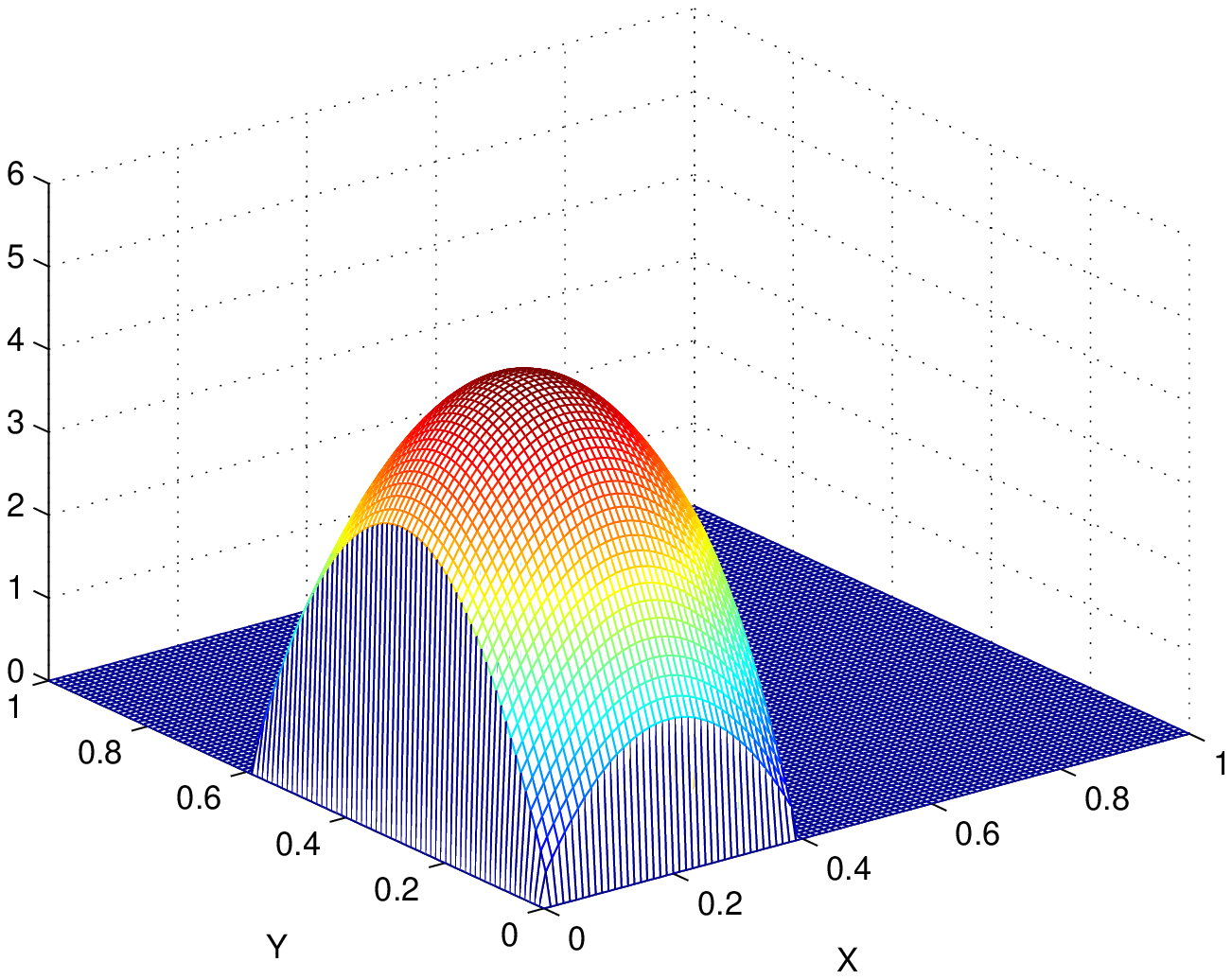,width=6.8cm}
 \end{minipage}\hfill \begin{minipage}[t]{.5\linewidth}
\centering\epsfig{figure=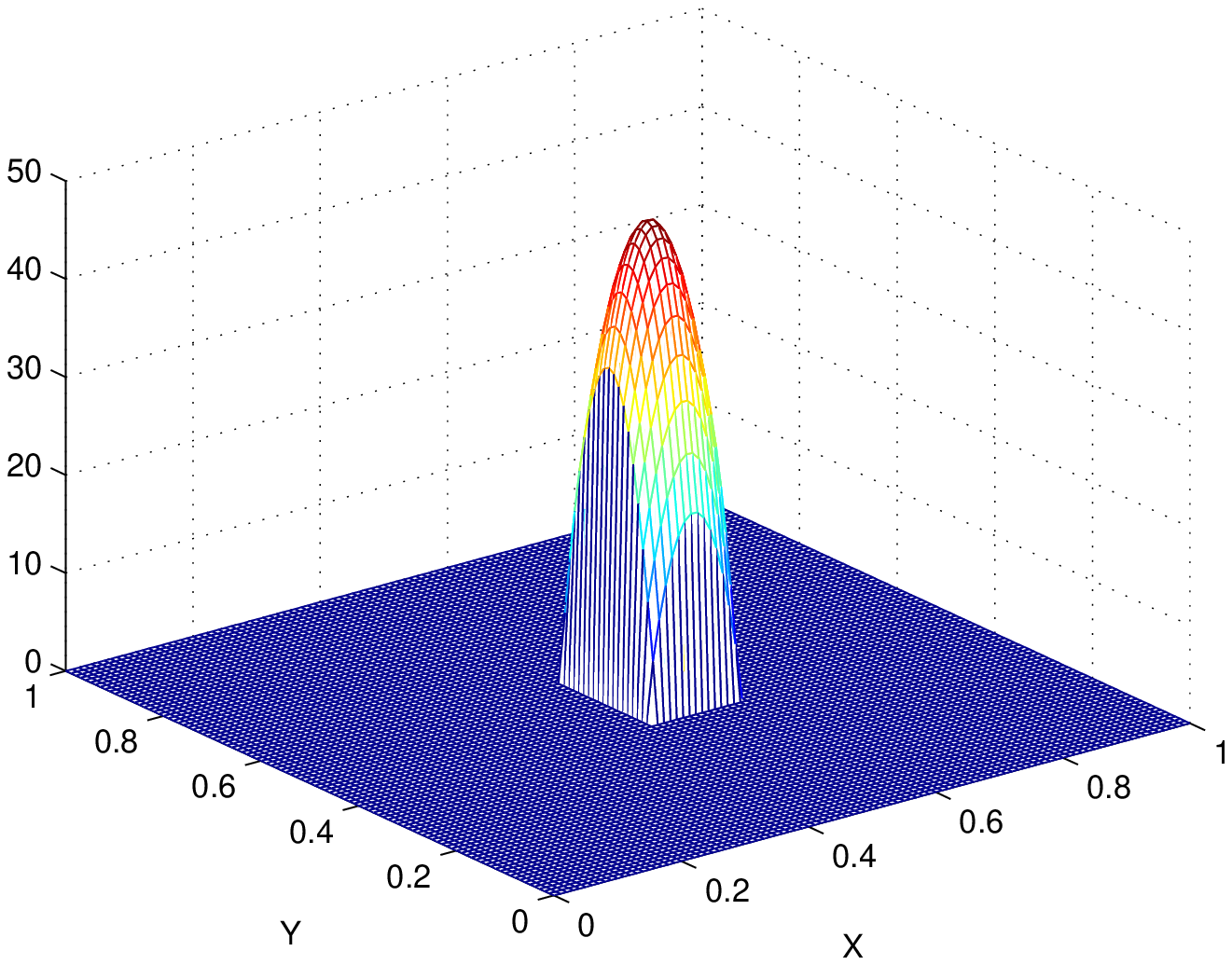,width=6.8cm}
 \end{minipage}
   \caption{\small Simulations obtained by best-reply iteration for $V_0(x,y)=(x-0.6)^2 + (y-0.7)^2$, $\phi(x,y)=|x-y|^4$ and $\eps=1/10$: the distribution of type $\mu$ on the left, distribution $\nu$ of the agents' action on the right.} \label{fig:2d}
\end{figure}
%-----------------
%%%%%%%%%%%%%%%%%%%%%%%%%%%%%%%
\section{Combining the variational approach with a fixed-point argument}\label{combivar}
\subsection{Symmetric interactions: a potential game approach}
In~\cite{abgc}, we obtain Cournot-Nash equilibria by a variational approach related to optimal transport. As already recalled in Section~\ref{sepcase}, under the separable form~\pref{sep}, if $\gamma$ is a Cournot-Nash equilibrium and $\nu$ denotes its second marginal then $\gamma\in \Pi_o(\mu, \nu)$ {\it i.e.} it solves the optimal transport problem:
\begin{equation}
\W_c(\mu, \nu):=\inf_{\gamma\in \Pi(\mu, \nu)} \int_{X\times Y} c(x,y) \dd \gamma(x,y)\;.
\end{equation}
If, in addition, externalities take the typical form
\[
\V[\nu](y)=f(y, \nu(y))+  \I[\nu](y) \quad\mbox{with}\quad \I[\nu](y)=\int_Y \phi(y,z) \dd\nu(z)
\]
with $f(y,.)$ increasing satisfying the growth condition~\pref{growth} and $\phi$ is continuous and \emph{symmetric} {\it i.e.} $\phi(y,z)=\phi(z,y)$, then we can associate to $\V[\nu]$ the functional
\[
\E[\nu]=\int_Y F(y,\nu(y))\dd m_0(y)+\frac{1}{2} \int_{Y\times Y} \phi(y,z) \dd\nu(y) \dd\nu(z)\;.
\]
In this setting, $\V$ is the first variation of $\E$, $\V[\nu] ={\delta \E}/{\delta \nu}$, in the sense that for every $(\rho, \nu)\in  \D^2$, we have
\[
\lim_{\eps\to 0^+} \frac{\E[(1-\eps)\nu+\eps\rho]-\E[\nu]}{\eps}=\int_{Y} \V[\nu] \dd(\rho-\nu)\;.
\]
It is therefore natural to consider the variational problem
\begin{equation}\label{varequil}
\inf_{\nu\in \D} \J_\mu[\nu]\quad\mbox{where}\quad \J_\mu[\nu]:=\W_c(\mu, \nu)+\E[\nu]\;.
\end{equation}
We assume:\\
{\bf{(H)}}: $X=\clos{\Omega}$ where $\Omega$ is some open bounded connected subset of $\R^d$ with negligible boundary, $\mu$ is equivalent to the Lebesgue measure on $X$ and, for every $y\in Y$, $c(.,y)$ is differentiable with $\nabla_x c$ bounded on $X\times Y$.

Under Assumption~{\bf{(H)}}, $\W_c(\mu, \nu)$ is G\^ateaux-differentiable with respect to $\nu$. It is not hard to check that the first-order optimality condition for~\pref{varequil} actually gives Cournot-Nash equilibria, see~\cite[Section~4]{abgc} for details. Moreover the assumptions above on $f$ and $\phi$ guarantee the existence of a minimiser, see~\cite[Theorem~4.3]{abgc}, and lead to:
%-----------------------------
\begin{thm}[Minimizers are equilibria]\label{miniareeq}
Assume that {\bf{(H)}} holds,  that $f(y,.)$ is increasing, satisfies the growth condition~\pref{growth} and that $\phi$ is symmetric and continuous. If $\nu$ solves~\eqref{varequil} and $\gamma$ solves  $\W_c(\mu, \nu)$ then $\gamma$ is a Cournot-Nash equilibrium. In particular there exist Cournot-Nash equilibria. 
\end{thm}
%-----------------------------
In other words, the situation described above may be related to potential games. The main drawback of Theorem~\ref{miniareeq} lies in the symmetry assumption for the interaction term~$\phi$. Symmetry is essential for $\V$ to have a potential but assuming symmetry may not particularly realistic, we shall see in the next section how to cope with more general non-symmetric interactions. 
\subsection{Existence for non-symmetric interactions}
In this section, consider $\V[\nu]$ be the sum of a local congestion term and a regular term:
\begin{equation}\label{congereg}
\V[\nu](y):=f(y, \nu(y))+ \I[\nu](y)\;.
\end{equation}\label{growth2}
We assume that $f(y,.)$ is increasing, satisfies the power growth condition: for some $\alpha >0$ and $C>0$ 
\begin{equation}
\forall (y,t) \qquad \frac{1}{C}  (t^{\alpha}-1)   \le  f(y, t) \le C (t^{\alpha}+1)\;.
\end{equation}
We also assume that $\I[\nu] \in \C(Y)$ for every $\nu\in \P(Y)$ with
\begin{equation}\label{REG2}
\nu \mapsto \I[\nu] \mbox{ is a continuous map from $(\P(Y),\textrm{w}\!-\!*)$ to $(\C(Y),\|\cdot\|_\infty)$.}
\end{equation}

This framework covers the case of a general pairwise interaction term 
\[
\I[\nu](y):=\int_Y \phi(y,z) \dd\nu(z)
\]
or more generally
\[
\I[\nu](y):=\int_Y \phi(y,z_1, \cdots, z_n) \dd\nu(z_1) \cdots \dd\nu(z_n)
\]
with an arbitrary continuous $\phi$. In this setting we have
%---------------------------
\begin{thm}[Existence of equilibria: non-symmetric interaction case]\label{existvar} Assume that~{\bf{(H)}} holds and that $\V$ has the form~\pref{congereg}. If $f(y,.)$ increasing and satisfies the growth condition~\pref{growth2} and $\I$ satisfies~\pref{REG2} then there exists at least one Cournot-Nash equilibrium. 
\end{thm}
%---------------------------
\begin{proof}
Let $p=\alpha+1$ and $K$ be the set of $\L^p$ probability densities. For $\nu\in K$ let us consider the minimisation problem
\[
\inf_{\theta\in K} \left\{\W_c(\mu, \theta)+ \int F(y, \theta(y)) \dd y+\int \I[\nu] \dd\theta\right\} 
\]
where $F(y,.)$ is a primitive of $f(y,.)$. By standard lower semi-continuity arguments, this problem has at least a solution that is in fact unique by strict convexity of $F(y,.)$ and convexity of the other terms. Let us denote by $G(\nu)$ this minimiser. It is easy to check that~\pref{REG2} implies that the map $G$ is continuous with respect to the weak topology of $\L^p$. Moreover, the growth condition~\pref{growth2} implies that $G(K)$ is bounded in $\L^p$ and hence relatively compact for the weak topology of $\L^p$. Thanks to Schauder's fixed-point theorem, there exists $\nu\in K$ such that $\nu=G(\nu)$. Writing the optimality condition we straightforwardly see, e.g.~\cite[Proof of Theorem~3.2]{abgc},  that if $\gamma$ solves $\W_c(\mu, \nu)$ then $(\gamma, \nu)$ is actually a Cournot-Nash equilibrium. 
\end{proof}
%%%%%%%%%%%%%
\subsection{An ordinary differential equation for  equilibria in dimension one} 
We now consider the one-dimensional case where $X=Y=[0,1]$ (say), $m_0$ is the Lebesgue measure on $X$, $\mu$ is equivalent to the Lebesgue measure and the cost $c \in \C^2$ satisfies the Spence-Mirrlees condition:
\[
\partial^2_{xy} c(x,y)<0\;.
\]
We again consider a separable total cost of the form 
$$
c(x,y)+f(\nu(y))+\int_Y \phi(y,z) \;\nu(z) \dd m_0(z)
$$
with $f$ increasing and $\phi$ continuous (and not necessarily symmetric). Replacing the interaction term $$\int_Y \phi(y,z) \nu(z) \dd z$$ by a more general of the form $$H\left(y, \int_Y \phi(y,z_1, \ldots, z_n) \nu(z_1) \dd m_0(z_1)\ldots  \nu(z_n) \dd m_0(z_n)\right)$$ actually costs no generality but we will not consider this case for sake of simplicity. For the clarity of the exposition, we focus on the congestion cost $f$ of the form:
\[
f(\nu)=\log(\nu) \quad \mbox{ or }\quad f(\nu)=\nu^\alpha, \mbox{ with } \alpha \ge 1 \;.
\]
As shown in~\cite{abgc}, in the case $f(\nu)=\log(\nu)$, the Inada condition holds which guarantees that $\nu$ is positive everywhere on $[0,1]$. This needs not be the case when $f(\nu)=\nu^\alpha$, with $\alpha \ge 1$. In both cases, because of the Spence-Mirrlees condition, by Corollary~\ref{purityofCNE} we know that equilibria are pure {\it i.e.} if $(\gamma, \nu)$ is an equilibrium then $\gamma=(\id, T)_\#\mu$ for some map $T$ which is the optimal transport between $\mu$ and $\nu$. This map is well-known to be the unique non-decreasing map which transports $\mu$ to $\nu$. In dimension one, this optimal map $T$ is easy to compute, once $\nu$ is known: it is indeed given by the formula $T=F_\nu^{-1} \circ F_\mu$ where $F_\mu$ is the cumulative distribution function of $\mu$ and $F_\nu^{-1}$ is the quantile function of $\nu$. Finding an equilibrum $(\gamma, \nu)$ thus amounts to find the transport map $T$ which as we shall see is characterised by some non-linear and non-local ordinary differential equation.

The equilibrium condition~\ref{equipp} can be rewritten as
\begin{multline}
\min_{x\in [0,1]} \{ c(x,y)-\varphi(x)\} +f(\nu(y))+\int_0^1 \phi(y, z) \;\nu(z) \dd z\\
= \varphi^c(y)+f(\nu(y))+\int_0^1 \phi(y, z)\; \nu(z) \dd z\ge 0 \label{equ1}
\end{multline}
with an equality for $y=T(x)$ which is the point which realises the minimum above, {\it i.e.} 
\[\varphi(x)=c(x,T(x))-\varphi^c(T(x))= \min_{y\in [0,1]} \{c(x,y)-\varphi(y)\}\;.\]
The smoothness of $c$ implies that $\varphi$ is Lipschitz hence differentiable a.e.. For a point of differentiability of $\varphi$, the envelope theorem gives
\begin{equation}\label{equ2}
\varphi'(x)=\partial_x c(x,T(x)) \quad\mbox{and hence}\quad \varphi(x)=\varphi(0)+\int_0^x \partial_x c(s, T(s)) \dd s
\end{equation}
\subsubsection{The logarithmic case}
In the case $f(\nu)=\log(\nu)$, as already mentioned, $\nu$ is positive everywhere on $[0, 1]$. So that $T$ is increasing on $[0,1]$, $T(0)=0$ and $T(1)=1$. By~\pref{equ1}-\pref{equ2} and the fact that  $T_\#\mu=\nu$, we obtain
\[\begin{split}
\nu(T(x))&= \exp\left(-\varphi^c(T(x))-\int_{0}^1 \phi(T(x),z )\;\nu(z) \dd z\right)\\
&=\exp\left(\varphi(x)-c(x,T(x))-\int_{0}^1 \phi(T(x), T(y) ) \dd\mu(y) \right)\\
&=\exp\left(\varphi(0)+\int_0^x \partial_x c(s, T(s)) \dd s-c(x,T(x))-\int_{0}^1 \phi(T(x), T(y) ) \dd\mu(y) \right).
\end{split}\]
Now, the fact that $T_\#\mu=\nu$ can be expressed as 
\begin{equation}\label{ma1d}
\mu(x)=\nu(T(x)) \;T'(x)\;.
\end{equation}
Replacing and setting $C:={\rm e}^{-\varphi(0)}$ we find the following equation on $T$:
\begin{equation}\label{odeforT1}
T'(x)=C \mu(x) \exp\left(-\int_0^x \partial_x c(s, T(s)) \dd s+c(x,T(x))+\int_{0}^1 \phi(T(x), T(y) ) \dd\mu(y) \right)
\end{equation}
supplemented with the initial condition $T(0)=0$ and $T(1)=1$. Since $T(1)=1$ the constant $C$ is given by
\[\frac{1}{C}= \int_0^1 \exp\left(-\int_0^x \partial_x c(s, T(s)) \dd s+c(x,T(x))+\int_{0}^1 \phi(T(x), T(y) ) \dd\mu(y) \right) \dd\mu(x) \;.\]
This gives the following easy to implement iterative algorithm:

\paragraph{Iterative algorithm  1: logarithmic congestion:}
Consider a given $T_k$ increasing with $T_k(0)=0$, $T_k(1)=1$.
\begin{itemize}
\item Define $C_k$ as being the inverse of 
\[
\int_0^1 \exp\left(-\int_0^x \partial_x c(s, T_k(s)) \dd s+c(x,T_k(x))+\int_{0}^1 \phi(T_k(x), T_k(y) ) \dd\mu(y) \right) \dd\mu(x)\;.
\]
\item Define $S_k$ as being 
\[
C_k \,\mu(x)\, \exp\left(-\int_0^x \partial_x c(s, T_k(s)) \dd s+c(x,T_k(x))+\int_{0}^1 \phi(T_k(x), T_k(y) ) \dd\mu(y) \right)\;.
\]
\item Then $T_{k+1}$ is given by
\[
T_{k+1}(x):=\int_0^x S_k(s) \dd s\;.
\]
\end{itemize}
See Figure~\ref{fig:log} for an example of such an implementation.
%----------------------
\begin{figure}[ht!]
 \begin{minipage}[t]{.5\linewidth}
\centering\epsfig{figure=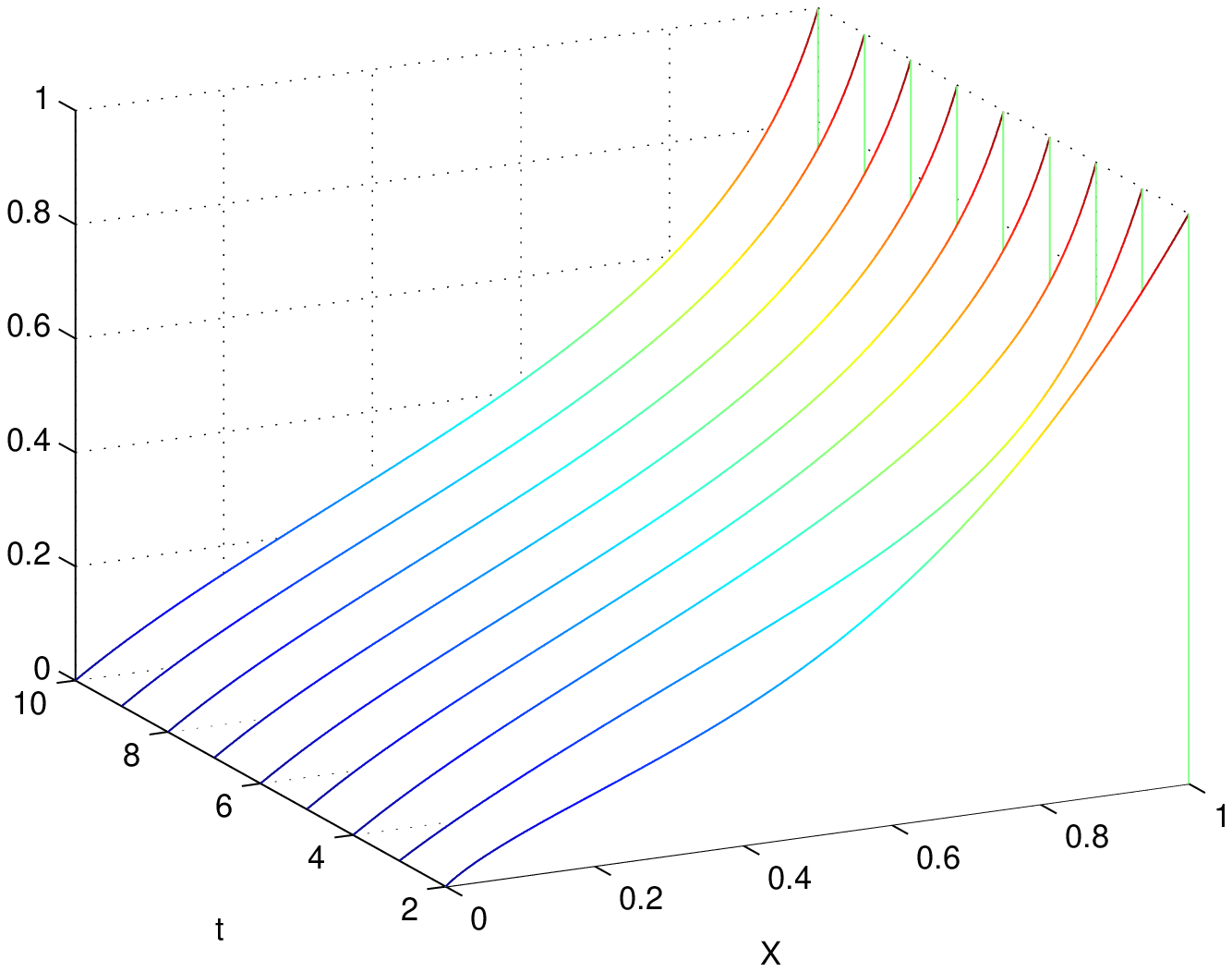,width=6.8cm}
 \end{minipage}\hfill \begin{minipage}[t]{.5\linewidth}
\centering\epsfig{figure=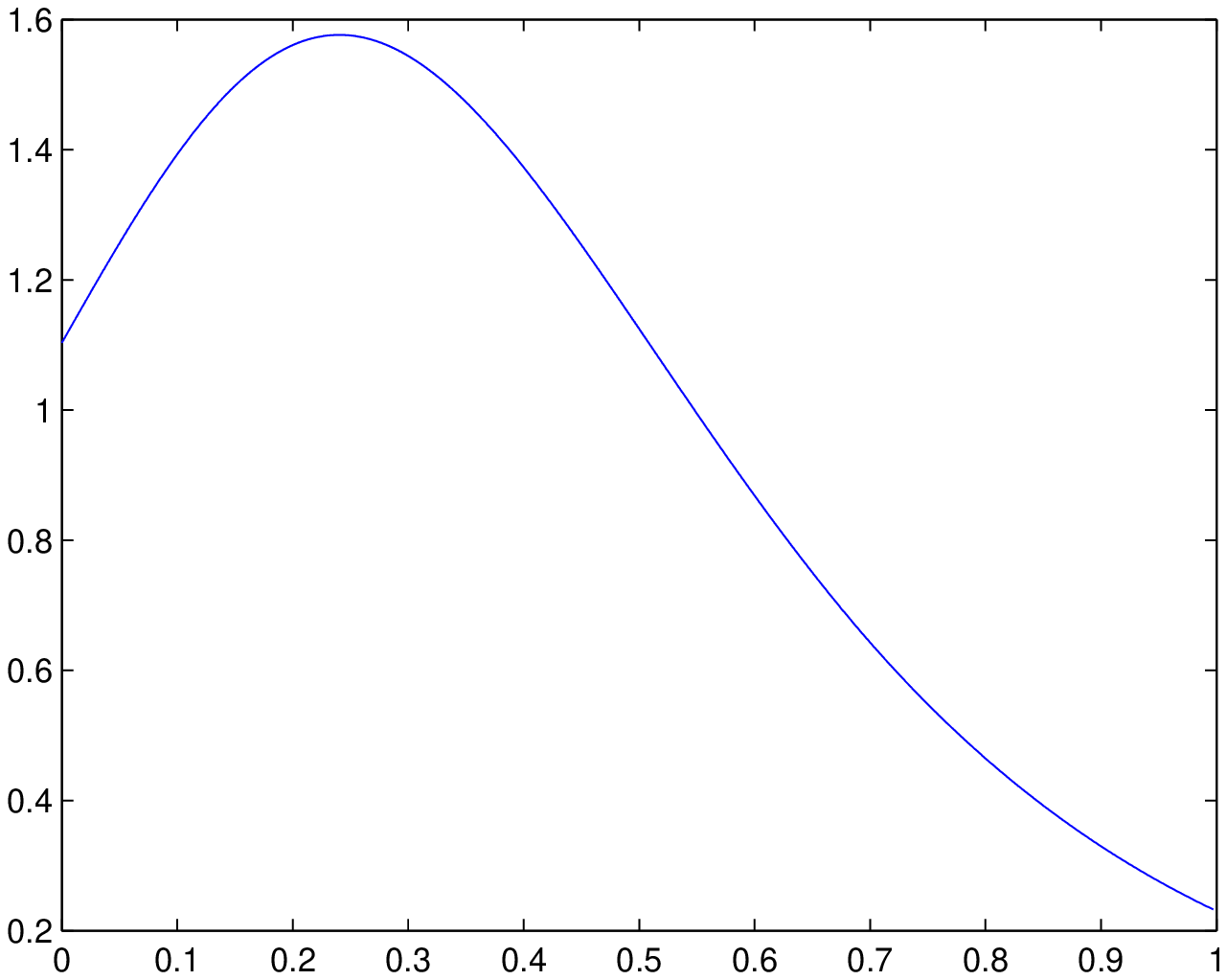,width=6.8cm}
 \end{minipage}
   \caption{\small Log case: Convergence for the iterates of algorithm 1 above for the transport on the left, and the density $\nu$ at the equilibrium on the right, in the case of a uniform $\mu$, $c(x,y)=\vert x-y\vert^{2.2}/2.2$ and a non-symmetric interaction given by $\phi(x,y)=2 \vert 3x/2-y \vert^{1.2}$.\label{fig:log}}
\end{figure} 
%----------------------
\subsubsection{Linear or power case}
Let us now consider the case where $f(\nu)=\nu^\alpha$, with $\alpha \ge 1$. The equilibrium condition can then be written as 
$$\nu(y)^\alpha+ \varphi^c(y)+ \int_0^1 \phi(y,z) \dd\nu(z) \ge \lambda\,,$$
for some constant $\lambda$, with an equality whenever $\nu(y)>0$. This condition can be rewritten as 
\[
\nu(y)=\left(\lambda- \varphi^c(y)-\int_0^1 \phi(y,z) \dd\nu(z)  \right)_+^{1/\alpha}\;.
\]
Since $\nu$ may vanish, $T$ may be discontinuous and the situation is actually more involved than in the log case. % (one cannot use an ODE for $T$ but just its integrated form $F_\nu\circ T=F_\mu$). 
 Actually, it is better to look for the optimal transport between $\nu$ and $\mu$ which may have flat zones but is continuous. This transport is given by
\[
S=F_\mu^{-1}\circ F_\nu\;.
\]
The integration constant is contained in the $\lambda$ above so that we can normalise to $\varphi^c(0)=0$. We also have, as before,
\[
\varphi^c(y)=\int_0^y \partial_y c(S(s), s)) \dd s\;.
\]
This leads to the following iterative algorithm.
\paragraph{Iterative algorithm  2: linear or power congestion:}
Let us start with a probability density $\nu_k$ on $[0,1]$, then:
\begin{itemize}
\item[$\bullet$] Define the optimal transport between $\nu_k$ and $\mu$:
\[S_{k}=F_{\mu}^{-1} \circ F_{\nu_k}\;,\]
where $F_{\mu}^{-1}$ has to be computed only once,
\item[$\bullet$] Compute the Kantorovich  potential $\varphi_k^c$ by
\[\varphi_k^c(y)=\int_0^y \partial_y c(S_k(s), s) \dd s\;,\]
\item[$\bullet$] Compute the new density $\nu_{k+1}$ by 
\[\nu_{k+1}(y)=\left(\lambda_k- \varphi_k^c(y)-\int_0^1 \phi(y,z) \dd\nu_k(z)  \right)_+^{1/\alpha}\;,\]
where $\lambda_k$ is such that $\nu_{k+1}$ has total mass $1$.
\end{itemize}
%------------------------
See Figure~\ref{fig:done} for an example of implementation of this algorithm.
%------------------------
\begin{figure}[ht!]
 \begin{minipage}[t]{.5\linewidth}
\centering\epsfig{figure=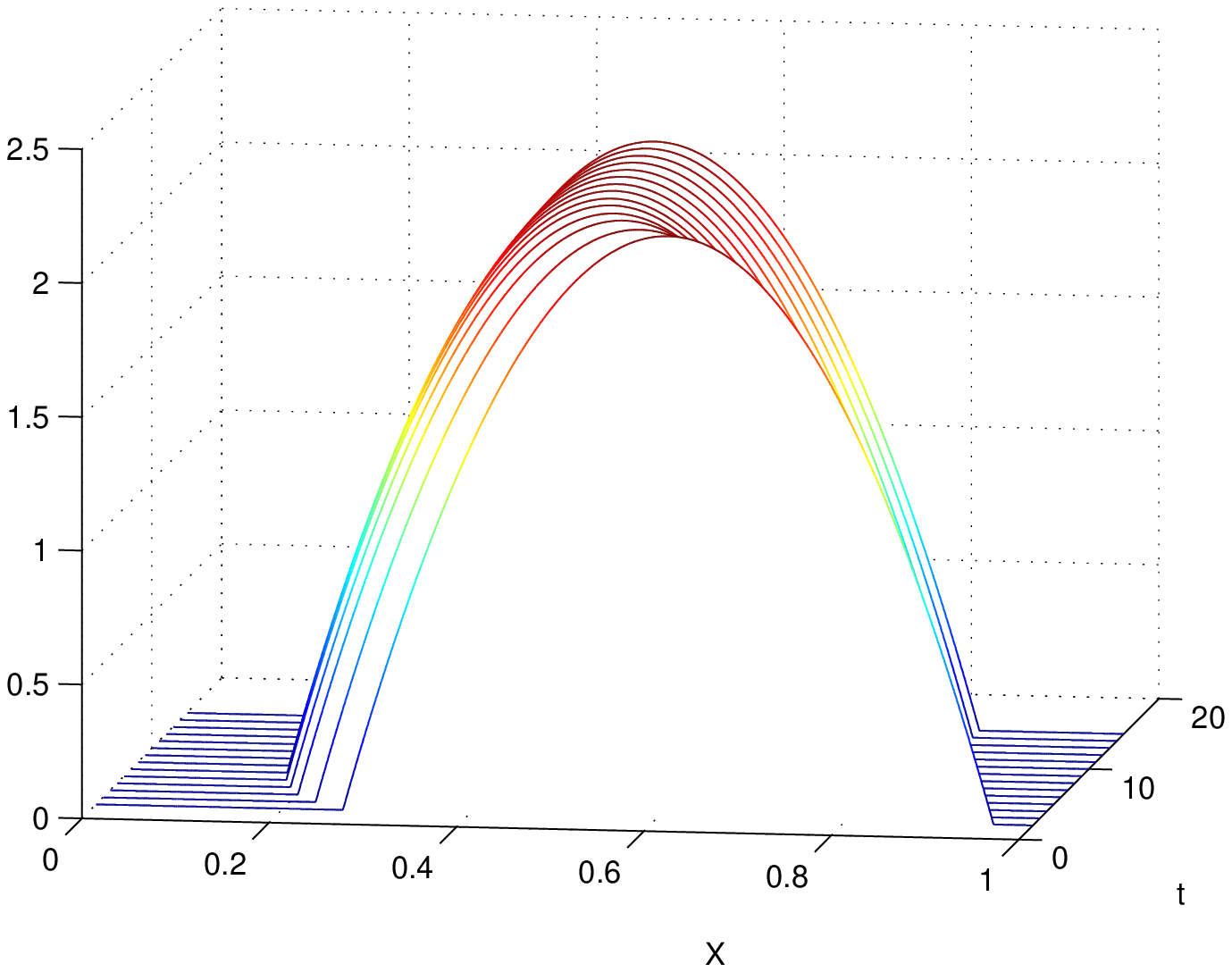,width=6.8cm}
 \end{minipage}\hfill \begin{minipage}[t]{.5\linewidth}
\centering\epsfig{figure=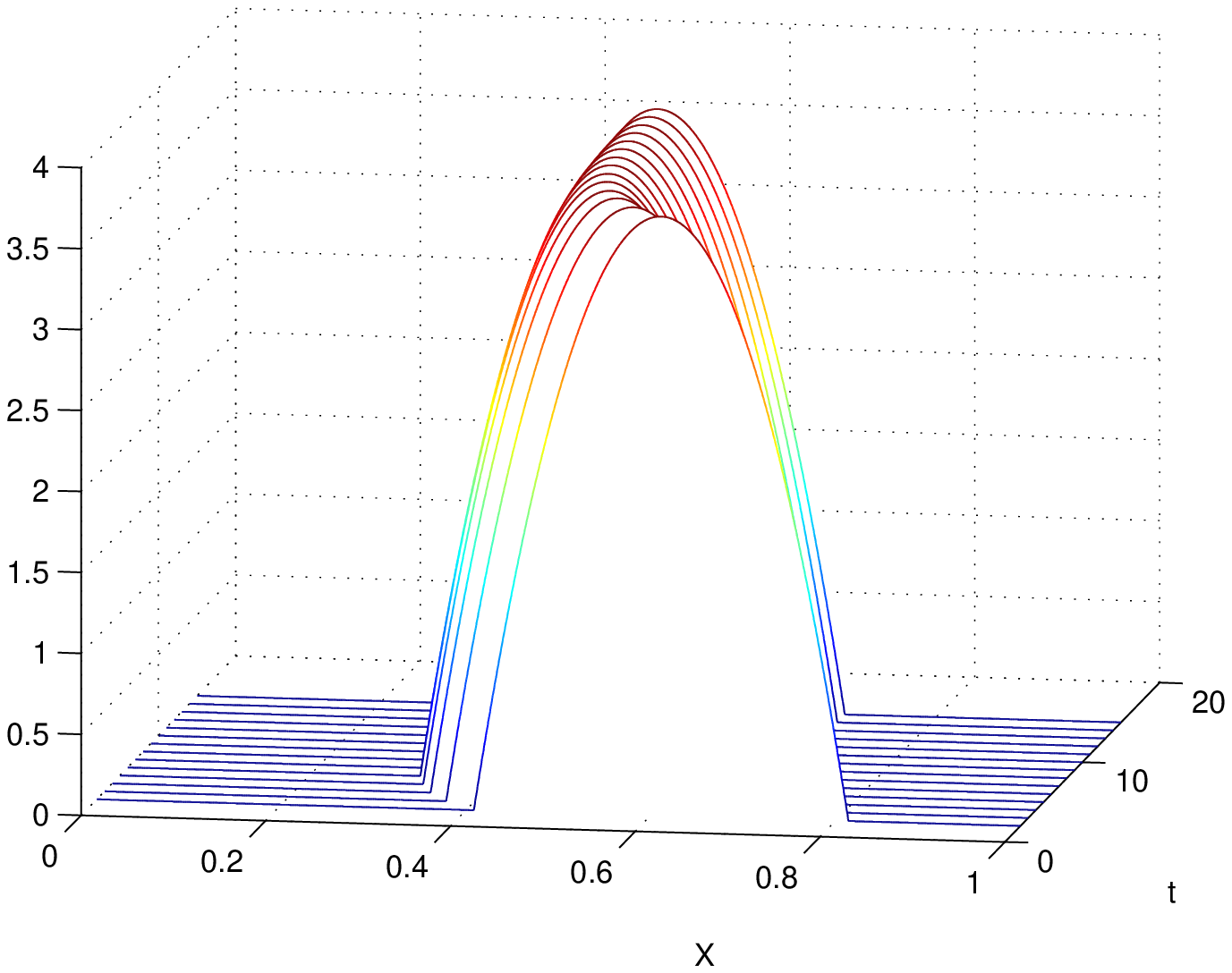,width=6.8cm}
 \end{minipage}
\caption{\small Linear case: convergence for the iterates of algorithm 2 above,   in the case of a uniform $\mu$, $c(x,y)=\vert x-y\vert^{4}/4$ and  interaction given by $\phi(x,y)=3 \vert 3x-2y -1/2\vert^{2}$. On the right, similar example with a non-symmetric interaction term given by $\phi(x,y)=10 \vert 3x-2y-1/2\vert^2$.\label{fig:done}} 
  \end{figure}
%------------------------

%%%%%%%%%%%%%%%%%%%%%%%%%%%%%%%%%%%%%%%%%
%%%%%%%%%%%%%%%%%%%%%%%%%%%%%%%%%%%%%%%%%
{\bf{Acknowledgements.}} The authors gratefully acknowledge the support of INRIA and the ANR through the Projects ISOTACE (ANR-12-MONU-0013) and OPTIFORM (ANR-12-BS01-0007).
%%%%%%%%%% Insert bibliography here %%%%%%%%%%%%%%
%%%%%%%%%%%%%%%%%%%%%%%%%%%%%%%%%%%%%%%%%%%%%%%%%%%%%%%%%%
%%% Bibliography %%%%%%%%%%%%%%%%%%%%%%%%%%%%%%%%%%%%%%%%%
%%%%%%%%%%%%%%%%%%%%%%%%%%%%%%%%%%%%%%%%%%%%%%%%%%%%%%%%%%
\bibliographystyle{siam}
\bibliography{biblio}
\end{document}